\documentclass{amsart}

\usepackage{amsmath,amssymb,amsthm}
\usepackage{comment,enumerate}
 \usepackage[capitalise]{cleveref}
\usepackage[all]{xy}

\crefname{equation}{Equation}{Equations}
\crefname{section}{Section}{Sections}
\crefname{chapter}{Chapter}{Chapters}

\newtheorem{lem}{Lemma}[section]
\newtheorem{thm}[lem]{Theorem}
\newtheorem*{thm*}{Theorem}

\newtheorem{prop}[lem]{Proposition}

\theoremstyle{definition}
\newtheorem{dfn}[lem]{Definition}

\newtheorem{questions}[lem]{Questions}

\newtheorem{rmk}[lem]{Remark}

\newcommand{\ZZ}{\mathbb{Z}}

\newcommand{\QQ}{\mathbb{Q}}

\newcommand{\PP}{\mathbb{P}}

\newcommand{\oper}[1]{\operatorname{#1}} 
\newcommand{\p}{\mathbb{P}^1}
\newcommand{\z}{\mathbb{Z}}

\newcommand{\q}{\mathbb{Q}}

\newcommand{\cha}{\operatorname{char}}

\newcommand{\e}{\'{e}}

\newcommand{\tw}[1]{{\widetilde{#1}}}

\newcommand{\benum}{\begin{enumerate}[(a)]}
\newcommand{\eenum}{\end{enumerate}}

\numberwithin{equation}{section}

\begin{document}

\title{Arithmetic descent of specializations of Galois covers}


\author{Ryan Eberhart}
\address{Department of Mathematics, The Pennsylvania State University, University Park, PA 16802, USA}
\email{rde3@psu.edu}
\thanks{The authors would like to thank Akshay Venkatesh for helpful discussions.}

\author{Hilaf Hasson}
\address{Department of Mathematics, Stanford University, Palo Alto, CA 94305, USA}
\email{hilaf@stanford.edu}


\date{}

\dedicatory{}

\commby{}

\begin{abstract}
Given a $G$-Galois branched cover of the projective line over a number field $K$, we study whether there exists a closed point of $\mathbb{P}^1_K$ with a connected fiber such that the $G$-Galois field extension induced by specialization ``arithmetically descends'' to $\mathbb{Q}$ (i.e., there exists a $G$-Galois field extension of $\mathbb{Q}$ whose compositum with the residue field of the point is equal to the specialization). We prove that the answer is frequently positive (whenever $G$ is regularly realizable over $\mathbb{Q}$) if one first allows a base change to a finite extension of $K$. If one does not allow base change, we prove that the answer is positive when $G$ is cyclic. Furthermore, we provide an explicit example of a Galois branched cover of $\mathbb{P}^1_K$ with no \it $K$-rational \rm points of arithmetic descent.
\end{abstract}

\maketitle

\section{Introduction}\label{introsect}
In this paper we introduce and study the notion of \it arithmetic descent \rm (see \cref{E:arth}) of specializations of $G$-Galois branched covers. Arithmetic descent is closely related to the Inverse Galois Problem (IGP), which asks whether for every finite group $G$ there exists a $G$-Galois field extension of $\mathbb{Q}$.

An easy argument (see for example \cite[Corollary 3.3.5]{msri}) using Riemann's Existence Theorem implies that for every finite group $G$ there exists a $G$-Galois branched cover (see \cref{galoisdfn}) of $\p_{\bar \q}$. By descent, this implies that there is a $G$-Galois branched cover over \it some \rm number field $K$. By Hilbert's Irreducibility Theorem (\cite[Theorem 13.4.2]{fieldarith}) there exist infinitely many $G$-Galois field extensions of $K$ arising from specializations of this cover at $K$-rational points. Most efforts to solve the IGP over $\mathbb{Q}$ have been aimed at producing a $G$-Galois branched cover of $\p_{\bar{\mathbb{Q}}}$ that descends to $\mathbb{Q}$ and then specializing. (Especially important is the technique of \it rigidity\rm , developed by Belyi, Feit, Fried, Shih, Thompson, Matzat and others; see \cite{vo}, \cite{serretopics}, and \cite{matz} for details. There has also been research aimed at understanding how the \it field of moduli \rm depends on the group, as well as on 
topological data; see for example \cite{beck}, \cite{flon}, \cite{wew}, \cite{obus1}, \cite{obus2}, and \cite{hass}.) In this paper we study an alternative method of attack.

By the above, for every finite group $G$ there exists a $G$-Galois branched cover of $\p_K$ for some number field $K$. We now ask whether a point on $\p_K$ can be chosen cleverly so that the $G$-Galois field extension coming from specialization descends to $\mathbb{Q}$, even if the cover itself does not. For example, consider the cyclic $C_4$-Galois cover $\p_{\mathbb{Q}(i)}\rightarrow \p_{\mathbb{Q}(i)}$ given by $y^4=x$. This cover does \it not \rm descend (together with its Galois action) to $\mathbb{Q}$. However, by specializing at $x=i$ we get the $C_4$-Galois extension $\mathbb{Q}(\zeta_{16})/\mathbb{Q}(i)$ (where $\zeta_{16}$ is a primitive $16^{\oper{th}}$ root of unity), which \it does \rm arithmetically descend to the $C_4$-Galois extension $\mathbb{Q}(\zeta_{16}+\zeta_{16}^{-1})/\mathbb{Q}$ (i.e., the compositum $\mathbb{Q}(\zeta_{16}+\zeta_{16}^{-1})\mathbb{Q}(i)$ is equal to $\mathbb{Q}(\zeta_{16})$).

There are several natural questions concerning arithmetic descent, which we now discuss. Let $G$ be a finite group, let $L/K$ be a finite extension of Hilbertian fields (examples of Hilbertian fields include number fields and function fields; see \cref{E:hilbertian}), and let $X\rightarrow \p_L$be a $G$-Galois branched cover. In this situation, we ask the following questions:
\begin{questions}\label{E:quest}\
\begin{enumerate}
\item Does there exist an \it $L$-rational point \rm $P$ of $\p_L$ such that the specialization at $P$ arithmetically descends to $K$?\label{q1}
\item Does there exist a \it closed point \rm $P$ of $\p_L$, not necessarily $L$-rational, such that the specialization at $P$ arithmetically descends to $K$?\label{q2}
\item Does Question (1) hold after finite base change? In other words, does there exist a finite field extension $E/L$ and an $E$-rational point $P$ of $\p_{E}$ such that the specialization at $P$ in $X_{E}\rightarrow \p_{E}$ arithmetically descends to $K$?\label{q3}
\end{enumerate}
\end{questions}
Note that a positive answer to Question \ref{E:quest}.\ref{q1} implies a positive answer to Question \ref{E:quest}.\ref{q2}, which implies a positive answer to Question \ref{E:quest}.\ref{q3}.

In the remainder of this paper, we attack all three questions. In Section \ref{E:defnques} we introduce the definitions we require. In Section \ref{E:sec1} (\cref{E:th1}) we show that the answer to Question \ref{E:quest}.\ref{q3} is frequently positive--whenever $G$ is regularly realizable over $K$. While this motivates further investigation, the proof of \cref{E:th1}
lacks geometric flavor, and does not lend itself to generalization.
In Section \ref{E:sec2} (\cref{nonthincor}) we answer Question \ref{E:quest}.\ref{q2} positively for cyclic groups of order coprime to $\cha(K)$.
In Section \ref{E:sec3} (\cref{nopointsthm}) we exhibit an explicit cover for which the answer to Question \ref{E:quest}.\ref{q1} is negative. Thus, our results seem to suggest that while Question \ref{E:quest}.\ref{q3} is the most likely to have a positive answer, Question \ref{E:quest}.\ref{q2} is the most geometrically interesting of the above questions.

\section{Definitions}\label{E:defnques}
\begin{dfn}\rm\label{galoisdfn}
Let $G$ be a finite group, let $K$ be a field, and let $X$ and $Y$ be geometrically irreducible varieties over $K$. Then a map $X\rightarrow Y$ of $K$-varieties is called a \textit{branched cover} (sometimes shortened to \textit{cover}) if it is finite and generically \e tale. A branched cover is \textit{Galois} if $\oper{Aut}(X/Y)$ acts freely and transitively on its geometric fibers, away from the ramification. A Galois branched cover is called \textit{G-Galois} if $\oper{Aut}(X/Y)$ is isomorphic to $G$. If there exists a $G$-Galois branched cover of $\PP^1_K$, we say that $G$ is \textit{regularly realizable} over $K$.
\end{dfn}

\begin{dfn}\label{E:arth}
Let $G$ be a finite group, $L/K$ a field extension, and $E/L$ a $G$-Galois extension of fields. We say that $E/L$ \textit{arithmetically descends to $K$} if there exists a $G$-Galois field extension $F/K$ such that $F\otimes_K L\cong E$ as $L$-algebras.

If $X\rightarrow Y$ is a branched cover of curves over $L$ then we say that a closed point $P$ of $Y$ \textit{arithmetically descends to $K$} if the specialization at $P$ is connected, and the residue field extension arithmetically descends to $K$.
\end{dfn}

\begin{dfn}
Let $V$ be an irreducible variety over a field $K$. Following \cite{serretopics}, a subset $A$ of $V(K)$ is of \textit{type $C_1$} if there is a closed proper subset $W\subseteq V$ such that $A\subseteq W(K)$. A subset $A$ of $V(K)$ is of \textit{type $C_2$} if there is a geometrically irreducible variety $V'$ of the same dimension as $V$, and a generically surjective morphism $\pi:V'\rightarrow V$ of degree greater than one, such that $A\subseteq\pi(V'(K))$.

A subset $A$ of $V(K)$ is called \textit{thin} if it is contained in a finite union of sets of type $C_1$ and $C_2$. The variety $V$ is of \textit{Hilbert type} if $V(K)$ is non-thin.
\end{dfn}
\begin{rmk}
The definition of sets of type $C_2$ in \cite{serretopics} requires that $V'$ is only irreducible, rather than geometrically irreducible. However, by the remark after Definition 3.1.3 in \cite{serretopics}, the two definitions are equivalent.
\end{rmk}

While we are primarily interested in number fields, most of our results apply to a more general type of field called Hilbertian fields. (See \cite[Chapters 12 and 13]{fieldarith} for a detailed discussion.)
\begin{dfn}\label{E:hilbertian}
A field $K$ is called \textit{Hilbertian} if $\PP^1_K$ is of Hilbert type.
\end{dfn}

\begin{dfn}
Let $L/K$ be a field extension and $X$ a $K$-scheme. The notation $X_L$ will be used for the base change $X\times_K L$ of $X$ to $L$. Let $\pi:X_{L}\rightarrow X$ be the morphism induced by this base change. For the sake of brevity, $X_{L}(K)$ will refer to the set $\pi^{-1}(X(K))$ of $L$-points of $X_L$ coming from
$K$-points of $X$.
\end{dfn}

The notation $C_n$ will be used for a cyclic group of order $n$, and the notation $\zeta_n$ will refer to a coherent choice of a primitive $n^{\oper{th}}$ root of unity in a fixed separable closure of the field in question.

\section{Arithmetic descent after finite base change}\label{E:sec1}
The following shows that the answer to Question \ref{E:quest}.\ref{q3} is often affirmative:
\begin{thm}\label{E:th1}
In the situation of Questions \ref{E:quest}, assume that $G$ is regularly realizable over $K$. Then the answer to Question \ref{E:quest}.\ref{q3} is positive for any $L/K$ and any $G$-Galois branched cover $X\rightarrow \p_L$. In particular, the Regular Inverse Galois Problem (RIGP) over $K$ implies an affirmative answer to Question \ref{E:quest}.\ref{q3} for any $L/K$ and any $G$-Galois branched cover of $\PP^1_L$.
\end{thm}
\begin{proof}
Let $P$ be an $L$-rational point of $\p_L$ whose fiber in $X$ is connected, and let $F/L$ be the resulting $G$-Galois field extension. By assumption, there exists a $G$-Galois branched cover $Y\rightarrow \p_K$. Consider the base change of this cover to $F$. The set $\p_F(K)$ is non-thin by \cite[Proposition 3.2.1]{serretopics}, while the subset of $\p_F(F)$ of points whose fiber in $Y_F\rightarrow \p_F$ is disconnected is thin by \cite[Proposition 3.3.1]{serretopics}. Since the intersection of a non-thin set with the complement of a thin set is non-empty, there is a point $Q\in \p_F(K)$ whose fiber in $Y_F$ is connected. By specializing at the corresponding $K$-point of $\p_K$ in $Y\rightarrow \p_K$, the resulting $G$-Galois field extension $E/K$ is linearly disjoint with $F$ over $K$.

Note that $(E\otimes_K F)/L$ is a Galois extension of fields with group $G\times G$. Let $\Delta$ denote the subgroup $\{(g,g)|g\in G\}$ of $G\times G$ and let $L'$ be the subfield of $E\otimes_K F$ fixed by $\Delta$. Let $P_{L'}$ be the unique $L'$-point of $\p_{L'}$ lying above $P$. Then $P_{L'}$ is a point of arithmetic descent for $X_{L'}\rightarrow \p_{L'}$. Indeed, the specialization is the field extension $F\otimes_L L'$ of $L'$, which is isomorphic as an $L'$-algebra to $E\otimes_K F$, which, in turn, is isomorphic to $E\otimes_K L'$ since $\Delta\cap (G\times 1)=\Delta\cap (1\times G)=1$. 
\end{proof}

\section{Closed points of arithmetic descent for cyclic groups}\label{E:sec2}
In this section, our objective is to establish the following answer to Question \ref{E:quest}.\ref{q2}.

\begin{thm}\label{nonthincor}
Question \ref{E:quest}.\ref{q2} has a positive answer for any $G$-Galois branched cover $X\rightarrow\PP^1_L$, where $G$ is a cyclic group of order coprime to $\cha(K)$. In fact, one can replace $\p_L$ with any smooth curve over $L$.
\end{thm}

Roughly, our method of proof will be to show (\cref{nonthinthm}) that the set of points of arithmetic descent for a particularly simple $C_n$-Galois cover is non-thin. After employing a suitable Cartesian diagram and dealing with issues arising from the potential lack of a primitive $n^{\oper{th}}$ root of unity in $L$, the result will follow for a general $C_n$-Galois cover.

\begin{thm}\label{nonthinthm}
Let $K$ be a Hilbertian field, $n$ a positive integer coprime to $\cha(K)$, and $L=K(\zeta_n)$. Let $\p_L\rightarrow \p_L$ be the $C_n$-Galois cover defined by $y^n=x$. Then the set of $L$-rational points which arithmetically descend to $K$ for this cover is non-thin.
\end{thm}
\begin{proof}
We first remark that there are infinitely many non-isomorphic $C_n$-Galois field extensions of $L$ which arithmetically descend to $K$. Indeed, assume by contradiction that all such extensions of $L$ are contained in a finite extension $E/L$. By \cite[Lemma 16.3.4]{fieldarith} there exists a $C_n$-Galois branched cover $Z\rightarrow \p_K$. The subset $\p_E(K)$ of $\p_E(E)$ is non-thin (\cite[Proposition 3.2.1]{serretopics}), and the set of $E$-rational points whose fiber in $Z_E\rightarrow \p_E$ is connected is the complement of a thin set (\cite[Proposition 3.3.1]{serretopics}). Since the intersection of a non-thin set and the complement of a thin set is non-empty, it follows that there is a $C_n$-Galois extension of $E$ which arithmetically descends to $K$. In particular there is a $C_n$-Galois extension of $L$ that arithmetically descends to $K$ which is not contained in $E$, a contradiction.

By the above, let $\{L_i\}_{i\in I}$ be an infinite set of non-isomorphic $C_n$-Galois field extensions of $L$, each of which arithmetically descends to $K$. For each $i$ in $I$ there exists an $a_i$ in $L$ such that $L_i=L(\sqrt[n]{a_i})$.
Let $f_i:\PP^1_L\rightarrow\PP^1_L$ be given by $x\mapsto a_i x^n$. Note that by Kummer theory $f_i(\PP^1_L(L)\smallsetminus\{0,\infty\})$ is the set of $L$-rational points whose specialization in the cover $\p_L\rightarrow \p_L$ given by $y^n=x$ is isomorphic to $L(\sqrt[n]{a_i})/L$. Therefore it suffices to show that $S=\bigcup f_i(\PP^1_L(L))$ is non-thin.

Suppose that $S$ is thin. By \cite[p.\ 245 Claim B(f)]{fieldarith} any set of type $C_1$ is contained in a set of type $C_2$, so we have that $$S \subseteq \bigcup_{j=1}^n g_i(Y_j(L))$$
where each $Y_j$ is a geometrically irreducible curve over $L$ and each $g_j:Y_j\rightarrow\PP^1_L$ has degree greater than one. Let $Y$ denote the disjoint union of the curves $Y_j$ and $g:Y\rightarrow \PP^1_L$ the map induced from the $g_j$'s.

Fix an index $i$. After taking the reduced structure if necessary, let $W$ denote the fiber product in the following diagram
$$\xymatrix{
W:=Y\times_{\PP^1_L}\PP^1_L\ar[d]\ar[r]^{\,\,\,\,\,\,\,\,\,\,\,\,\,\,\,\,\,\,\,\,\pi} & \PP^1_L\ar[d]^{f_i} \\
Y\ar[r]^g & \PP^1_L
}$$
Since $f_i(\PP^1_L(L))\subseteq g(Y(L))$, we have that $\PP^1_L(L)\subseteq \pi(W(L))$ and thus $\pi(W(L))$ is non-thin. Finiteness is base change invariant, so $\pi$ is finite, which implies that $W$ consists of finitely many irreducible components of dimension no greater than one. These facts collectively imply that there exists an irreducible component $W'$ of $W$ such that the restriction of $\pi$ to $W'$ has a section on an open subvariety $U$ of $\PP^1_L$. Thus for every index $i$ we have that $x\mapsto a_i x^n$ factors birationally through one of finitely many maps $g_j:Y_j\rightarrow\PP^1_L$. However, this is impossible since $f_i(\PP^1_L(L))\bigcap f_{i'}(\PP^1_L(L))=\{0,\infty\}$ for $i\neq i'$.
\end{proof}

Before proving \cref{nonthincor}, we require one last lemma.

\begin{lem}\label{TNTlem}
Let $L'/L$ be a finite extension of Hilbertian fields, and let $S\subseteq \PP^1_{L'}(L')$ be a non-thin set. Suppose that $S=\bigcup_{i\in I} S_i$ where each $S_i$ is equal to $f_i(\PP^1_{L'}(L))$ for a cover $f_i:\PP^1_{L'}\rightarrow\PP^1_{L'}$. If $R$ is a thin set, then there exists an index $i\in I$ such that $S_i\smallsetminus R$ is infinite.
\end{lem}

\begin{proof}
By \cite[p.\ 245 Claim B(f)]{fieldarith} any set of type $C_1$ is contained in a set of type $C_2$, so we have that $$R\subseteq \bigcup_{j=1}^n g_j(Z_j(L'))$$
where each $Z_j$ is a geometrically irreducible curve over $L'$ and each $g_j:Z_j\rightarrow\PP^1_{L'}$ has degree greater than one. Let $Y_j$ be the smooth projective model of $Z_j$. The map $g_j$, considered as a map on the smooth locus of $Z_j$, can be extended to a map $\tw{g_j}:Y_j\rightarrow\PP^1_{L'}$. Let $R'$ be the union $\bigcup \tw{g_j}(Y_j(L'))$.

Assume that $S_i\smallsetminus R'$ is finite. We claim that $S_i$ is then contained in $R'$. To this end, let $U$ be the open subscheme of $\PP^1_{L'}$ obtained by removing $f_i^{-1}(S_i\smallsetminus R')$. Let $Y$ be the disjoint union of the curves $Y_j$ and $g:Y\rightarrow\PP^1_{L'}$ the map induced by the $\tw{g_j}$'s. Since $f_i(U\cap\p_{L'}(L))$ is contained in $g(Y(L'))$, and since $U\cap\p_{L'}(L)$ is a non-thin set, a similar argument to the one in the proof of \cref{nonthinthm} shows that there exists a rational map $h$ such that the following commutes:

$$\xymatrix{
& U\ar[d]^{f_i}\ar@{-->}[ld]_h\\
Y\ar[r] & \PP^1_{L'}
}$$
Since $Y$ is proper, $h$ can be extended to a map defined on all of $\PP^1_{L'}$ and the claim follows.

Since $S$ is non-thin, there must exist an $i\in I$ such that $S_i$ is not contained in $R'$. By the above, $S_i\smallsetminus R'$ is then infinite. Because $R'$ contains all but finitely elements of R, $S_i\smallsetminus R$ is infinite as well.
\end{proof}

With the above results, we are now ready to prove \cref{nonthincor}.

\begin{proof}[Proof of \cref{nonthincor}]
Let $f:X\rightarrow Y$ be a $C_n$-Galois branched cover of curves over $L$. Let $L'=L(\zeta_n)$
and let $f_{L'}:X_{L'}\rightarrow Y_{L'}$ be the base change of $f$ to $L'$.
By Kummer theory $f_{L'}$ is given by $t^n=g$, where $g\in\kappa(Y_{L'})$. Observe the Cartesian diagram
$$\xymatrix{
X_{L'}\ar[r]\ar[d]_{f_{L'}} & \PP^1_{L'}\ar[d] \\
Y_{L'}\ar[r]_g & \PP^1_{L'}&\!\!\!\!\!\!\!\!\!\!\!\!\!\!\!\!\!\!\!\!\!\!\!\!\!\!\!\!\!\!\!\!,
}$$
where the bottom horizontal map is induced by $g$, the vertical right map is $y^n=x$, and the vertical left map is $f_{L'}$.

Since $L'$ is Hilbertian, the set $D$ of $L'$-points of $\p_{L'}$ whose fiber in the map $g\circ f_{L'}:X_{L'}\rightarrow \PP^1_{L'}$ is disconnected is a thin set. Indeed, let $L^{\oper{sep}}$ be a separable closure of $L'$, let $\bar X\rightarrow \p_{L^{\oper{sep}}}$ be the base change of $X_{L'}\rightarrow \p_{L'}$ to $L^{\oper{sep}}$, let $\bar Z\rightarrow \p_{L^{\oper{sep}}}$ be its Galois closure, and let $Z\rightarrow \p_E$ be a model of this Galois closure over some finite field extension $E$ of $L'$. Then by \cite[Proposition 3.3.1]{serretopics}, the set of $E$-points in $\p_E$ whose fiber in $Z$ is disconnected is thin. This implies, using \cite[Proposition 3.2.1]{serretopics}, that the subset of points of $\p_{L'}(L')$ whose fiber in $Z$ is disconnected is thin. In particular, the subset of points in $\p_{L'}(L')$ whose fiber in $X_{L'}$ is disconnected is thin.

Let $A\subseteq \PP^1_{K(\zeta_n)}(K(\zeta_n))$ be the set of points of arithmetic descent for the cover $y^n=x$. By \cref{nonthinthm}, $A$ is non-thin. By \cite[Proposition 3.2.1]{serretopics} $A$, considered as a subset of $\PP^1_{L'}(L')$, remains non-thin. Moreover, $A\smallsetminus D$, with $D$ defined as in the previous paragraph, is non-thin. For any $a\in A\smallsetminus D$, let $b$ be the unique preimage of $a$ under $g:Y_{L'}\rightarrow \p_{L'}$. By our choice of $a$, the $C_n$-Galois field extension $L'(\sqrt[n]{a})/L'$ arithmetically descends to $K$. The specialization of $f_{L'}$ at $b$ is then $(L'(\sqrt[n]{a})\otimes_{L'} \kappa(b))/\kappa(b)$, where $\kappa(b)$ is the residue field of $b$. This specialization is a field extension since the fiber of $g\circ f_{L'}$ over $b$ is connected. Therefore, $b$ is a point of arithmetic descent to $K$ for the cover $f_{L'}$.

Note that if $L'$ is equal to $L$ (i.e., if $L$ already contains a primitive $n^{\oper{th}}$ root of unity), then this concludes the proof. We may therefore assume that $L'$ is a non-trivial extension of $L$. Let $\pi:Y_{L'}\rightarrow Y_L$ be the base change map. Let 
$$B=\{a\in A|\zeta_n\notin\kappa(\pi(z)) \mbox{ for every }z \mbox{ such that } g(z)=a \}.$$ 
It suffices to show that $A\smallsetminus (B\cup D)$ is non-empty. Indeed, for any $a$ in $A\smallsetminus (B\cup D)$, the closed point $\pi(g^{-1}(a))$ is then a point of arithmetic descent.

Define an equivalence relation on $A$ by $a_1\sim a_2$ if and only if $a_1/a_2$ is an $n^{\oper{th}}$ power in $L$. We proceed to show that only finitely many elements in each equivalence class of $A$ are contained in $B$. Indeed, let $a$ be an element of $A$, and let $c$ be an element of $L$ such that $c^na$ is in $B$. Then by the definition of $B$ there is a $z\in Y_{L'}$ such that $g(z)=c^na$ and the residue field $\kappa(\pi(z))$ of $\pi(z)$ does not contain $L'$. After fixing an $L$-embedding of $\kappa(\pi(z))$ into $L^{\oper{sep}}$, this implies that $\kappa(\pi(z))\cap L'$ is a proper subextension $F$ of $L'/L$. Note that this subextension is independent of the embedding, because every subextension of $L'/L$ is Galois over $L$. It suffices to show that only finitely many such $z$ can produce the same $F$.

To this end, we first show that $g$ cannot be of the form $g=c\cdot h$ where $c$ is in $L'$ and $h$ is in $\kappa(Y_F)$. Indeed, let $y^n=h$ define a cover $X'\rightarrow Y_F$ of curves over $F$. Since this map is a model over $F$ of the cover $X_{L^{\oper{sep}}}\rightarrow Y_{L^{\oper{sep}}}$ given by $y^n=g$, by \cite[Lemma 4.16]{hilaf2} $X'\rightarrow Y_F$ is a ``twist'' (in the sense defined in \cite{hilaf2}) of $X_F\rightarrow Y_F$. Furthermore, since $C_n$ is abelian $X'\rightarrow Y_F$ must be Galois by \cite[Lemma 4.18]{hilaf2}. However, $y^n=h$ cannot be Galois over $F$ since $F$ does not contain a primitive $n^{\oper{th}}$ root of unity. (Note that since we have assumed that $X$ is geometrically irreducible, it follows that $g$ is not an $n^{\oper{th}}$ power in $\kappa(Y_{L^{\oper{sep}}})$, and so $h$ is not an $n^{\oper{th}}$ power in $\kappa(Y_F)$.)

Let $k$ be the degree of $L'/F$, and write $g=\sum_{i=0}^{k-1} g_i\zeta_n^i$ and $a=\sum_{i=0}^{k-1}a_i\zeta_n^i$, where each $g_i$ is in $\kappa(Y_F)$ and each $a_i$ is in $F$. Since $g$ is not an $L'$-multiple of an element in $\kappa(Y_F)$, there are indices $i$ and $j$ such that $g_i/g_j$ is a non-constant function. By assumption $\kappa(\pi(z))\cap L'=F$, and so we have that $g_i(z)=c^na_i$ and $g_j(z)=c^na_j$. Since only finitely many points map via $g_i/g_j$ to $a_i/a_j$, only finitely many $z$'s satisfy that $\kappa(\pi(z))\cap L'$ is equal to $F$, as claimed.

Let $a$ be an element of $A$. Note that the equivalence class of $a$ in $A$ is precisely the image of $\p_{L'}(L)$ for the map $\PP^1_{L'}\rightarrow\PP^1_{L'}$ defined by $x\mapsto a x^n$. Therefore, \cref{TNTlem} (with $A$ in the role of $S$, the equivalence classes of $A$ in the role of the $S_i$'s, and $D$ in the role of $R$) implies that there is an equivalence class of $A$ which has infinitely many points not in $D$. Since each such equivalence class of $A$ has only finitely many points in $B$, the set $A\smallsetminus (B\cup D)$ is non-empty, as we wanted to prove.
\end{proof}


\section{A cover with no rational points of arithmetic descent}\label{E:sec3}
The objective of this section is to provide an example of a cover for which the answer to Question \ref{E:quest}.\ref{q1} is negative. To be precise, in this section we prove the following.
\begin{thm}\label{nopointsthm}
Let $z$ be a parameter for $\PP^1_{\QQ(\zeta_3)}$. Then the $C_3$-Galois cover of $\PP^1_{\QQ(\zeta_3)}$ defined by $t^3=3(z^3+2)$ has no $\QQ(\zeta_3)$-rational points that arithmetically descend to $\QQ$.
\end{thm}

In order to prove this theorem, we first give a complete characterization of the points of arithmetic descent of a particularly simple $C_3$-Galois cover.

\begin{prop}\label{E:saltforthree}
Let $T$ be the set of $\mathbb{Q}(\zeta_3)$-rational points which arithmetically descend to $\QQ$ for the $C_3$-Galois cover $\p_{\mathbb{Q}(\zeta_3)}\rightarrow \p_{\mathbb{Q}(\zeta_3)}$ given by $t^3=z$. Let $C$ be the set of cubes in $\QQ(\zeta_3)$. Then $T=\{(x+\zeta_3y)^2 (x+\zeta_3^2y)|x,y\in \mathbb{Q}\}\smallsetminus C$.
\end{prop}
\begin{proof}
Let $a=(x+\zeta_3 y)^2 (x+\zeta_3^2 y)$ where $x,y\in\mathbb{Q}$, and let $\sqrt[3]{a}$ be a third root of $a$ in the specialization at $z=a$. Note that the specialization at $z=a$ is disconnected precisely when $a\in C$. Let $\sigma \in \oper{Gal}(\mathbb{Q}(\zeta_3,\sqrt[3]{a})/\mathbb{Q}(\zeta_3))$ be such that $\sigma(\sqrt[3]{a})=\zeta_3\sqrt[3]{a}$, and let $\tau$ be the generator of $\oper{Gal}(\mathbb{Q}(\zeta_3)/\mathbb{Q})$. Note that $\frac{\sqrt[3]{a}^2}{x+\zeta_3y}$ is a third root of $\tau(a)$. Indeed, $$\left(\frac{\sqrt[3]{a}^2}{x+\zeta_3y}\right)^3=\frac{a^2}{(x+\zeta_3y)^3}=(x+\zeta_3^2y)^2(x+\zeta_3y)=\tau(a).$$
Therefore, $\tau$ extends to an element in $\oper{Gal}(\mathbb{Q}(\zeta_3, \sqrt[3]{a})/\mathbb{Q})$ by taking $\sqrt[3]{a}$ to $\frac{\sqrt[3]{a}^2}{x+\zeta_3y}$. In order to show that $z=a$ is a point of arithmetic descent, we wish to show that $\sigma$ and $\tau$ commute. It suffices to check this on $\sqrt[3]{a}$:
$$\sigma(\tau(\sqrt[3]{a}))=\sigma\left(\frac{\sqrt[3]{a}^2}{x+\zeta_3y}\right)=\zeta_3^2\frac{\sqrt[3]{a}^2}{x+\zeta_3y}=\tau(\zeta_3)\frac{\sqrt[3]{a}^2}{x+\zeta_3y}=\tau(\zeta_3\sqrt[3]{a})=\tau(\sigma(\sqrt[3]{a})).$$

We now show the reverse direction. Let $a$ be in $T$, and let $\sqrt[3]{a}$ be a third root of $a$ in the specialization at $z=a$. As before, let $\sigma \in \oper{Gal}(\mathbb{Q}(\zeta_3,\sqrt[3]{a})/\mathbb{Q}(\zeta_3))$ be such that $\sigma(\sqrt[3]{a})=\zeta_3\sqrt[3]{a}$, and let $\tau$ be the generator of $\oper{Gal}(\mathbb{Q}(\zeta_3)/\mathbb{Q})$. Since $z=a$ is a point of arithmetic descent, the automorphism $\tau$ extends to an element in $\oper{Gal}(\mathbb{Q}(\zeta_3, \sqrt[3]{a})/\mathbb{Q})$ that commutes with $\sigma$. Therefore, $$\sigma(\tau(\sqrt[3]{a}))=\tau(\sigma(\sqrt[3]{a}))=\tau(\zeta_3\sqrt[3]{a})=\zeta_3^2\tau(\sqrt[3]{a}).$$ Since $\sigma(\sqrt[3]{a}^2)=\zeta_3^2\sqrt[3]{a}$, it follows that $\sigma(\frac{\tau(\sqrt[3]{a})}{\sqrt[3]{a}^2})=\frac{\tau(\sqrt[3]{a})}{\sqrt[3]{a}^2}$. Therefore $\frac{\tau(\sqrt[3]{a})}{\sqrt[3]{a}^2}$ is equal to an element $b$ in $\mathbb{Q}(\zeta_3)$. Since $\tau$ is an involution, it follows that $$\sqrt[3]{a}=\tau^2(\sqrt[3]{a})=\tau(b\sqrt[3]{a})
=\tau(b)\tau(\sqrt[3]{a}^2)=\tau(b)b^2\sqrt[3]{a}^4=\tau(b)b^2a\sqrt[3]{a}.$$
Therefore $a=\frac{1}{\tau(b)b^2}=\tau(\frac1b)(\frac1b )^2$. Taking $x,y\in\mathbb{Q}$ so that $\frac1b = x+\zeta_3y$, we see that $a=(x+\zeta_3 y)^2(x+\zeta_3^2 y)$.
\end{proof}
\begin{rmk}\rm
That every point in $T$ is a point of arithmetic descent is a rephrasing of an observation in \cite{albert}, which, together with its generalization in \cite{miki}, applies to every power of an odd prime. Those results, together with an analogous result for powers of $2$ (first appearing in \cite{sueyoshi}), are summarized neatly in Section 2 of \cite{saltman} (and in particular Theorems 2.3 and 2.4).

The reverse direction does not appear in these papers, but is most closely related to Theorem 2.3.b in \cite{saltman}. Indeed, following the proof one sees that Saltman proves a slightly stronger statement than Theorem 2.3.b. Namely, in his notation, he proves that $a^k$ is in the image of $M_{\tau}$. In the situation of Theorem 2.3, if $q=3$ then it's possible to choose $m=2$, in which case $k=1$, and the proposition above follows.
\end{rmk}

Throughout this section, we let $g(x,y)=(x+\zeta_3y)^2(x+\zeta_3^2y)$. By \cref{E:saltforthree}, if $g(x,y)=3(z^3+2)$ has no solution where $x,y\in\QQ$ and $z\in \QQ(\zeta_3)$, then the cover $t^3=3(z^3+2)$ has no $\QQ(\zeta_3)$-rational point of arithmetic descent, except possibly at infinity. We shall prove that this is the case by showing that this equation has no solution in the completion at the unique prime $(\pi)$ above $(3)$ in $\QQ(\zeta_3)$, where $\pi=1+2\zeta_3$. 

We first collect the requisite technical lemmas.

\begin{lem}\label{pidivideslem} \label{removepilem}
Let $x$ and $y$ be elements in $\mathbb{Z}$ such that $\pi$ divides $g(x,y)=(x+\zeta_3y)^2(x+\zeta_3^2y)$. Then $\pi$ divides both $(x+\zeta_3y)$ and $(x+\zeta_3^2y)$. Moreover, there exist $x',y'\in\ZZ$ such that $g(x',y')=g(x,y)/\pi^3$.
\end{lem}

\begin{proof}
The first part of the lemma follows easily from the fact that $\z[\zeta_3]$ is a PID, and that $x+\zeta_3^2y$ is the Galois conjugate of $x+\zeta_3y$.

Let $N:\QQ(\zeta_3)\rightarrow\QQ$ denote the field norm. For any $a,b\in\ZZ$, we have that $g(a,b)=(a+\zeta_3b)N(a+\zeta_3b)$. Since $\pi$ divides $g(x,y)$, by the first part of this lemma $\pi$ divides $(x+\zeta_3 y)$. Therefore, there exist $x',y'\in\ZZ$ such that $(x'+\zeta_3 y')=-(x+\zeta_3 y)/\pi$. We have that
\begin{align*}
g(x',y')&=(x'+\zeta_3 y')N(x'+\zeta_3 y')=\frac{-(x+\zeta_3 y)}{\pi}\cdot\frac{N(-1)N(x+\zeta_3 y)}{N(\pi)}\\
&=\frac{-(x+\zeta_3 y)N(x+\zeta_3 y)}{-\pi^3}=\frac{g(x,y)}{\pi^3},
\end{align*}
proving the second part of the lemma.
\end{proof}

A key step in the proof of \cref{nopointsthm} is reducing the question of finding solutions to $g(x,y)=3(z^3+2)$ to finding solutions in the finite ring $\ZZ[\zeta_3]/(81)$. After this reduction, we shall make use of the following two lemmas.

\begin{lem}\label{sage1lem}
Consider $g(x,y)$ as a polynomial in $(\ZZ[\zeta_3]/(81))[x,y]$ and fix $c\in\ZZ[\zeta_3]/(81)$. Then the set $\{c^3\cdot g(x,y)|x,y \in \ZZ/(81)\}$ is contained in $\{g(x,y)|x,y \in \ZZ/(81)\}$.
\end{lem}
\begin{proof}
This is a finite checking problem. We implemented an algorithm in SAGE \cite{sage} which verifies this result. The interested reader may consult \cite{sagelems} for the code.
\end{proof}

\begin{lem}\label{sage2lem}
The equation $g(x,y)=3(z^3+2)$ has no solution with $x,y\in \ZZ/(81)$ and $z\in\ZZ[\zeta_3]/(81)$.
\end{lem}
\begin{proof}
This is also a finite checking problem, which we again verify using SAGE \cite{sage}. The interested reader may consult \cite{sagelems} for the code.
\end{proof}

We remark that we reduce modulo 81 since it is the first power of 3 for which \cref{sage2lem} holds. Equipped with the above lemmas, we can now prove the main result of this section.

\begin{proof}[Proof of \cref{nopointsthm}]
By the discussion at the beginning of this section, we need to show that there is no solution to $g(x,y)=3(z^3+2)$ where $x,y\in\QQ$ and $z\in\QQ(\zeta_3)$, and in addition we need to show that the point at infinity does not arithmetically descend to $\QQ$. 

Assume by contradiction that we have such a solution to $g(x,y)=3(z^3+2)$. Let $x=x_0/x_1$, $y=y_0/y_1$, and $z=z_0/z_1$, where $x_i,y_i\in\ZZ$ and $z_i\in\ZZ[\zeta_3]$. Clearing denominators, we have that
$(z_1)^3 g(x_0 y_1,y_0 x_1)=3((z_0 x_1 y_1)^3+2(x_1 y_1 z_1)^3)$. Therefore it suffices to show that there is no solution to the equation

\begin{equation}\label{counterexampleeqn}
a^3g(x,y)=3(z^3+2c^3)
\end{equation}
with $x,y\in\ZZ$, $z\in\ZZ[\zeta_3]$, and non-zero $a,c\in\ZZ[\zeta_3]$.

Case 1: $\pi$ does not divide $c$. Reduce \cref{counterexampleeqn} modulo $(81)=(\pi)^8$. We denote the reduction by the use of bars over the elements. Since $\pi$ does not divide $c$, $\bar{c}^3$ is a unit in $\ZZ/(81)$ and so we may divide by it. Letting $z'=\bar{z}/\bar{c}$, the resulting equation may be written as $(\bar{a}/\bar{c})^3g(\bar{x},\bar{y})=3((z')^3+2)$. Since $(\bar{a}/\bar{c})^3$ is a cube, there are $x',y'\in\ZZ/(81)$ such that $g(x',y')=(\bar{a}/\bar{c})^3g(\bar{x},\bar{y})$ by \cref{sage1lem}. Therefore $g(x',y')=3((z')^3+2)$, in contradiction to \cref{sage2lem}.

Case 2: $\pi$ divides $c$. From this solution we shall produce a solution where $\pi$ does not divide $c$, which we have shown is impossible in Case 1 of the proof. Since $\pi$ divides the right hand side of \cref{counterexampleeqn} (because $\pi$ divides 3), it must divide either $a$ or $g(x,y)$. Due to \cref{pidivideslem}, it follows that $\pi^3$ must divide the left hand side of \cref{counterexampleeqn}. Since $\pi^3$ does not divide $3$, $\pi$ must divide $z^3+2c^3$. Since $\pi$ divides $z^3+2c^3$ and $c$, it must also divide $z^3$, and therefore $z$.

As mentioned, $\pi$ divides either $a$ or $g(x,y)$. Assume first that $\pi$ divides $a$. In this case, we may divide \cref{counterexampleeqn} by $\pi^3$ and get $(a/\pi)^3g(x,y)=3((z/\pi)^3+2(c/\pi)^3)$. Iterating this process yields a solution where either $\pi$ does not divide $c$ (in which case we are done by Case 1 of the proof), or $\pi$ divides $c$ but not $a$.

Therefore, we are left with the case where $\pi$ divides $g(x,y)$ but not $a$. By \cref{removepilem}, there are then $x',y'\in\ZZ$ such that $g(x',y')=g(x,y)/\pi^3$. Dividing both sides of \cref{counterexampleeqn} by $\pi^3$ yields the equation $a^3g(x',y')=3((z/\pi)^3+(c/\pi)^3)$. By iterating this process, we reduce to the case where $\pi$ does not divide $c$, as desired.

It remains to check that infinity does not arithmetically descend to $\QQ$. Homogenizing $t^3=3(z^3+2)$ with respect to $w$ gives us $t^3=3(z^3+2w^3)$. This equation defines a smooth curve in $\PP^2_{\QQ(\zeta_3)}$, giving a projective model for the cover. Above infinity (when $w=0$), the fiber is connected with residue field extension $\QQ(\zeta_3,\sqrt[3]3)/\QQ(\zeta_3)$. If the equation $g(x,y)=3$ has no solution where $x,y\in\QQ$, this field extension does not arithmetically descend to $\QQ$ by \cref{E:saltforthree}. One checks that $g(x,y)\in\QQ$ precisely when $y=0$. Since $g(x,0)=x^3$, there is no solution to $g(x,y)=3$ because 3 is not a cube in $\QQ$.
\end{proof}

\begin{rmk}
\cref{nonthincor} states that, although this cover has no $\QQ(\zeta_3)$-rational points of arithmetic descent, there are infinitely many \textit{closed} points which arithmetically descend to $\QQ$.
\end{rmk}

\bibliography{arithmetic_descent}{}
\bibliographystyle{amsalpha}
\end{document}